\documentclass[11pt]{article}

\usepackage{amssymb,amsmath,amsfonts,amsthm}
\usepackage[shortlabels]{enumitem}
\usepackage{latexsym}
\usepackage{graphics}
\usepackage{indentfirst}
\usepackage{tikz}
\usepackage{mathtools}
\usepackage{hyperref}
\allowdisplaybreaks

\newtheorem{innercustomthm}{Theorem}
\newenvironment{customthm}[1]
  {\renewcommand\theinnercustomthm{#1}\innercustomthm}
  {\endinnercustomthm}

\newtheorem*{thm*}{Theorem}
\newtheorem{thm}{Theorem}
\newtheorem{lem}[thm]{Lemma}
\newtheorem{pro}[thm]{Proposition}

\newtheorem{cor}[thm]{Corollary}

\newtheorem{ques}[thm]{Question}

\newcommand{\N}{\mathbb{N}}

\begin{document}

\title{A Note on Fractional DP-Coloring of Graphs}

\author{Daniel Dominik$^1$, Hemanshu Kaul$^2$, and Jeffrey A. Mudrock$^3$}

\footnotetext[1]{Department of Applied Mathematics, Illinois Institute of Technology, Chicago, IL 60616.  E-mail:  {\tt {ddominik@hawk.iit.edu}}}
\footnotetext[2]{Department of Applied Mathematics, Illinois Institute of Technology, Chicago, IL 60616.  E-mail:  {\tt {kaul@iit.edu}}}
\footnotetext[3]{Department of Mathematics and Statistics, University of South Alabama, Mobile, AL 36688. E-mail: {\tt mudrock@southalabama.edu}}

\maketitle

\begin{abstract}

DP-coloring (also called correspondence coloring) is a generalization of list coloring introduced by  Dvo\v{r}\'{a}k and Postle in 2015. In 2019, Bernshteyn, Kostochka, and Zhu introduced a fractional version of DP-coloring.  They showed that unlike the fractional list chromatic number, the fractional DP-chromatic number of a graph $G$, denoted $\chi_{_{DP}}^*(G)$, can be arbitrarily larger than $\chi^*(G)$, the graph's fractional chromatic number. We generalize a result of Alon, Tuza, and Voigt (1997) on the fractional list chromatic number of odd cycles, and, in the process, show that for each $k \in \N$, $\chi_{_{DP}}^*(C_{2k+1}) = \chi^*(C_{2k+1})$. We also show that for any $n \geq 2$ and $m \in \N$, if $p^*$ is the solution in $(0,1)$ to $p=(1-p)^n$ then $\chi_{_{DP}}^*(K_{n,m})\leq1/p^*$, and we prove a generalization of this result for multipartite graphs. Finally, we determine a lower bound on $\chi_{_{DP}}^*(K_{2,m})$ for any $m \geq 3$.  

\medskip



\end{abstract}

\section{Introduction}\label{intro}

In this paper all graphs are nonempty, finite, simple graphs unless otherwise noted.  Generally speaking we follow West~\cite{W01} for terminology and notation.  The set of natural numbers is $\N = \{1,2,3, \ldots \}$.  Given a set $A$, $\mathcal{P}(A)$ is the power set of $A$.  For $m \in \N$, we write $[m]$ for the set $\{1,2, \ldots, m \}$.  If $G$ is a graph and $S, U \subseteq V(G)$, we use $G[S]$ for the subgraph of $G$ induced by $S$, and we use $E_G(S, U)$ for the subset of $E(G)$ with one endpoint in $S$ and the other endpoint in $U$.  For $v \in V(G)$, we write $d_G(v)$ for the degree of vertex $v$ in the graph $G$, and we write $N_G(v)$ for the neighborhood of vertex $v$ in the graph $G$.  Also, for $S \subseteq V(G)$, we let $N_G(S) = \bigcup_{v \in S} N_G(v)$.  A graph $G$ is $d$-degenerate if every subgraph of $G$ has a vertex of degree at most $d$.  We use $K_{n,m}$ to denote the complete bipartite graphs with partite sets of size $n$ and $m$.  For a random variable $X$, we use $X\sim B(n,p)$ to indicate that $X$ is binomially distributed with $n$ trials each having probabilitiy of success $p$.  For an event $E$, we use $\overline{E}$ to denote the complement of $E$.

\subsection{Fractional Coloring and Fractional Choosability}

Before introducing fractional DP-coloring, we review some classical notions.  Given a graph $G$, in the classical vertex coloring problem we wish to color the elements of $V(G)$ with colors from the set $[m]$ so that adjacent vertices receive different colors, a so-called \emph{proper $m$-coloring}.  We say $G$ is \emph{$m$-colorable} when a proper $m$-coloring of $G$ exists.  The \emph{chromatic number} of $G$, denoted $\chi(G)$, is the smallest $k$ such that $G$ is $k$-colorable.

A \emph{set coloring} of a graph $G$ is a function that assigns a set to each vertex of $G$ such that the sets assigned to adjacent vertices are disjoint.  For $a, b \in \N$ with $a \geq b$, an \emph{$(a,b)$-coloring} of graph $G$ is a set coloring $f$ of $G$ such that the codomain of $f$ is the set of $b$-element subsets of $[a]$.  We say that $G$ is \emph{$(a,b)$-colorable} when an $(a,b)$-coloring of $G$ exists.  So, saying $G$ is $a$-colorable is equivalent to saying that it is $(a,1)$-colorable.  The \emph{fractional chromatic number}, $\chi^*(G)$, of $G$ is defined by $\chi^*(G) = \inf \{ a/b : \text{$G$ is $(a,b)$-colorable} \}.$  Since any graph $G$ is $(\chi(G),1)$-colorable, we have that $\chi^*(G) \leq \chi(G)$.  This inequality may however be strict; for example, when $r \geq 2$, $\chi^*(C_{2r+1}) = 2 + 1/r < 3 = \chi(C_{2r+1})$ (see~\cite{SU97}).  It is also well known that the infimum in the definition of $\chi^*(G)$ is actually a minimum~\cite{SU97}.

List coloring is a variation on classical vertex coloring that was introduced independently by Vizing~\cite{V76} and Erd\H{o}s, Rubin, and Taylor~\cite{ET79} in the 1970's.  In list coloring, we associate with graph $G$ a \emph{list assignment} $L$ that assigns to each vertex $v \in V(G)$ a list $L(v)$ of available colors. Graph $G$ is said to be \emph{$L$-colorable} if there exists a proper coloring $f$ of $G$ such that $f(v) \in L(v)$ for each $v \in V(G)$ (we refer to $f$ as a \emph{proper $L$-coloring} of $G$).  A list assignment $L$ is called a \emph{k-assignment} for $G$ if $|L(v)|=k$ for each $v \in V(G)$.  We say $G$ is \emph{k-choosable} if $G$ is $L$-colorable whenever $L$ is a $k$-assignment for $G$.  The \emph{list chromatic number} of $G$, denoted $\chi_\ell(G)$, is the smallest $k$ for which $G$ is $k$-choosable.  Since a $k$-assignment can assign the same $k$ colors to every vertex of a graph, $\chi(G) \leq \chi_\ell(G)$.

Given an $a$-assignment $L$ for graph $G$ and $b \in \N$ such that $a \geq b$, we say that $f$ is an \emph{$(L,b)$-coloring} of $G$ if $f$ is a set coloring of $G$ such that for each each $v \in V(G)$, $f(v) \subseteq L(v)$ with $|f(v)| = b$.  We say that $G$ is \emph{$(L,b)$-colorable} when an $(L,b)$-coloring of $G$ exists.  Also, for $a, b \in \N$ with $a \geq b$, graph $G$ is \emph{$(a,b)$-choosable} if $G$ is $(L,b)$-colorable whenever $L$ is an $a$-assignment for $G$.  The \emph{fractional list chromatic number}, $\chi_\ell^*(G)$, of $G$ is defined by $\chi_\ell^*(G) = \inf \{ a/b : \text{$G$ is $(a,b)$-choosable} \}.$  It is clear that if a graph is $(a,b)$-choosable, then it is $(a,b)$-colorable.  So, $\chi^*(G) \leq \chi_\ell^*(G)$.  In 1997, Alon, Tuza, and Voigt~\cite{AT97} famously proved that for any graph $G$, $\chi_\ell^*(G) = \chi^*(G)$.  Moreover, they showed that for any graph $G$, there is an $M \in \N$ such that $G$ is $(M, M/\chi^*(G))$-choosable.  So, the infimum in the definition of $\chi_\ell^*(G)$ is also actually a minimum.

In their 1979 paper Erd\H{o}s et al.~\cite{ET79} asked: If $G$ is $(a,b)$-choosable and $c, d \in \N$ are such that $c/d > a/b$, must $G$ be $(c,d)$-choosable?  A negative answer to this question is given in~\cite{GT09}.  Erd\H{o}s et al. also asked: If $G$ is $(a,b)$-choosable, does it follow that $G$ is $(at,bt)$-choosable for each $t \in \N$?  Tuza and Voigt~\cite{TV96} showed that the answer to this question is yes when $a=2$ and $b=1$.  However, in general, a negative answer to this question was recently given in~\cite{DH18}.It was shown that for each $a \geq 4$, a graph that is $(a,1)$-choosable but not $(2a,2)$-choosable can be constructed. We briefly consider the fractional DP-coloring analogues of both of these questions below.

\subsection{Fractional DP-coloring}

In 2015, Dvo\v{r}\'{a}k and Postle~\cite{DP15} introduced DP-coloring (they called it correspondence coloring) in order to prove that every planar graph without cycles of lengths 4 to 8 is 3-choosable.  Intuitively, DP-coloring is a generalization of list coloring where each vertex in the graph still gets a list of colors but identification of which colors are different can vary from edge to edge.  We now give the formal definition. Suppose $G$ is a graph.  A \emph{cover} of $G$ is a pair $\mathcal{H} = (L,H)$ consisting of a graph $H$ and a function $L: V(G) \rightarrow \mathcal{P}(V(H))$ satisfying the following four requirements:%We don't really follow [2] exactly because of (2)...

\vspace{5mm}

\noindent(1) the set $\{L(u) : u \in V(G) \}$ is a partition of $V(H)$ with $|V(G)|$ parts; \\
(2) for every $u \in V(G)$, the graph $L(u)$ is an independent set of vertices in $H$; \\
(3) if $E_H(L(u),L(v))$ is nonempty, then $uv \in E(G)$; \\
(4) if $uv \in E(G)$, then $E_H(L(u),L(v))$ is a matching (the matching may be empty).

\vspace{5mm}

Suppose $\mathcal{H} = (L,H)$ is a cover of $G$.  We say $\mathcal{H}$ is \emph{$m$-fold} if $|L(u)|=m$ for each $u \in V(G)$.  An $\mathcal{H}$-coloring of $G$ is an independent set $I\subseteq V(H)$ such that $|I \cap L(u)|=1$ for each $u \in V(G)$.  The \emph{DP-chromatic number} of a graph $G$, $\chi_{_{DP}}(G)$, is the smallest $m \in \N$ such that $G$ admits an $\mathcal{H}$-coloring for every $m$-fold cover $\mathcal{H}$ of $G$.

Given an $m$-assignment $L$ for a graph $G$, it is easy to construct an $m$-fold cover $\mathcal{H}$ of $G$ such that $G$ has an $\mathcal{H}$-coloring if and only if $G$ has a proper $L$-coloring (see~\cite{BK17}).  It follows that $\chi_\ell(G) \leq \chi_{_{DP}}(G)$.  This inequality may be strict since it is easy to prove that $\chi_{_{DP}}(C_n) = 3$ whenever $n \geq 3$, but the list chromatic number of any even cycle is 2 (see~\cite{BK17} and~\cite{ET79}).

We are now ready to introduce fractional DP-coloring.  Suppose $\mathcal{H} = (L,H)$ is an $a$-fold cover of $G$ and $b \in \N$ such that $a \geq b$.  Then, $G$ is \emph{$(\mathcal{H},b)$-colorable} if there is an independent set $S \subseteq V(H)$ such that $|S \cap L(v)| \geq b$ for each $v \in V(G)$. Equivalently, one could require $|S \cap L(v)| = b$ for each $v \in V(G)$; in this case we call $S$ an \emph{independent $b$-fold transversal}.  We refer to $S$ as an \emph{$(\mathcal{H},b)$-coloring} of $G$.  For $a,b \in \N$ and $a \geq b$, we say $G$ is \emph{$(a,b)$-DP-colorable} if for any $a$-fold cover $\mathcal{H}$ of $G$, $G$ is $(\mathcal{H},b)$-colorable.  The \emph{fractional DP-chromatic number}, $\chi_{_{DP}}^*(G)$, of $G$ is defined by
$$\chi_{_{DP}}^*(G) = \inf \{ a/b : \text{$G$ is $(a,b)$-DP-colorable} \}.$$
It is easy to prove that if $G$ is $(a,b)$-DP-colorable, then $G$ is $(a,b)$-choosable.  Also, any graph $G$ must be $(\chi_{_{DP}}(G),1)$-DP-colorable.  So, combining the facts we know, we have:
$$\chi^*(G) = \chi_{\ell}^*(G) \leq \chi_{_{DP}}^*(G) \leq \chi_{_{DP}}(G).$$
Theorem~\ref{thm: fractionalK2m} and Corollary~\ref{cor:completeMultipartite} below imply that both of the inequalities above can be strict.  Furthermore, we know that $\chi_{\ell}^*(G) \leq \chi_\ell(G) \leq \chi_{_{DP}}(G)$, and we will see below that it is possible for the list chromatic number of a graph to be either smaller ($K_{2,3}$ by Theorem~\ref{thm: fractionalK2m} below) or larger (odd cycles with at least five vertices by Theorem~\ref{thm: oddcycle} below) than the fractional DP-chromatic number of the graph.

In~\cite{BK18}, the following result is proven.
\begin{thm} [\cite{BK18}] \label{thm: fracDP2}
Let $G$ be a connected graph.  Then, $\chi_{_{DP}}^*(G) \leq 2$ if and only if $G$ contains no odd cycles and at most one even cycle.  Furthermore, if $G$ contains no odd cycles and exactly one even cycle, then $\chi_{_{DP}}^*(G)=2$ even though $2$ is not contained in the set $\{ a/b : \text{$G$ is $(a,b)$-DP-colorable} \}$.
\end{thm}
So, unlike the fractional chromatic number and fractional list chromatic number, the infimum in the definition of the fractional DP-chromatic number is not always a minimum.  In~\cite{BK18} it is also shown that if $G$ is a graph of maximum average degree $d \geq 4$, then $\chi_{_{DP}}^*(G) \geq d/(2 \ln d)$.  Since bipartite graphs have fractional chromatic number (and hence fractional list chromatic number) 2 and there exist bipartite graphs with arbitrarily high average degree, we see $\chi_{_{DP}}^*(G)$ and $\chi^*(G)$ can be arbitrarily far apart and $\chi_{_{DP}}^*(G)$ cannot be bounded above by a function of $\chi^*(G)$.

\subsection{Outline of Results and Open Questions}

We now present an outline of the results of this paper while also mentioning some open questions.  In Section~\ref{oddCycle} we study the fractional DP-chromatic number of odd cycles.  In 1997, Alon, Tuza, and Voigt~\cite{AT97} showed that $C_{2r+1}$ is $(2r+1,r)$-choosable.  We generalize this result by showing $\chi^*(C_{2r+1})= \chi_{_{\ell}}^*(C_{2r+1}) = \chi_{_{DP}}^*(C_{2r+1})$.

\begin{thm} \label{thm: oddcycle}
$C_{2r+1}$ is $(2r+1,r)$-DP-colorable.  Consequently, $\chi_{_{DP}}^*(C_{2r+1})=2 + 1/r$.
\end{thm}

Notice that by Theorem~\ref{thm: oddcycle}, we see it is possible for the list chromatic number of a graph to be larger than its fractional DP-chromatic number since $\chi_{_{DP}}^*(C_{2r+1}) < \chi_\ell(C_{2r+1}) = 3$ when $r \geq 2$. Other classes of graphs with this strict inequality are shown by Corollary~\ref{cor:completeMultipartite} and Theorem~\ref{thm: fractionalK2m} below.

It is natural to ask analogues of the two questions posed about $(a,b)$-choosability in~\cite{ET79}.

\begin{ques} \label{ques: Erdos1}
If $G$ is $(a,b)$-DP-colorable and $c, d \in \N$ are such that $c/d > a/b$, must $G$ be $(c,d)$-DP-colorable?
\end{ques}

\begin{ques} \label{ques: Erdos2}
If $G$ is $(a,b)$-DP-colorable, does it follow that $G$ is $(at,bt)$-DP-colorable for each $t \in \N$?
\end{ques}

Question~\ref{ques: Erdos2} is open.  We suspect the answer is no because of Dvo\v{r}\'{a}k, Hu, and Sereni's similar list coloring result~\cite{DH18}. The answer to Question~\ref{ques: Erdos1} is no in a fairly strong sense.  In particular, Corollary 1.12 in~\cite{BK18} implies that for any positive real number $c$, there exist $a,b,k \in \N$ such that $k - a/b > c$, $G$ is not $k$-colorable (and therefore not $(k,1)$-DP-colorable), and $G$ is $(a,b)$-DP-colorable\footnote[4]{Thanks to an anonymous referee for this observation.}.  By combining some known results, we quickly observe that the answer to Question~\ref{ques: Erdos1} remains no even if we restrict our attention to only bipartite graphs.

\begin{pro} \label{cor: examples}
For each $k \geq 149$, there exists a $k$-degenerate bipartite graph $G$, and $a, b \in \N$ such that: $k > a/b$, $G$ is not $(k,1)$-DP-colorable, and $G$ is $(a,b)$-DP-colorable.
\end{pro}

Proposition~\ref{cor: examples} is a consequence of two results.  It follows from the randomized construction given in the proof of Theorem~1.10 in~\cite{BK18} that for $d\geq149$,  if $G$ is a $d$-degenerate bipartite graph, then $\chi^*_{DP}(G) \leq 5d/\ln(d) < d$.  Note $K_{d,t}$ is a $d$-degenerate bipartite graph, and it was shown in~\cite{M18} that if $t \geq 1 + (d^d/d!)(\ln(d!) + 1)$, then $\chi_{_{DP}}(K_{d,t}) = d+1$.

    In Section~\ref{multipartiteGraphs} we study the fractional DP-chromatic number of multipartite graphs with a special focus on complete bipartite graphs.  We begin with a probabilistic proof of the following.\footnote[5]{An anonymous referee provided us an elegant proof of a special case of Corollary~\ref{cor:completeMultipartite} which follows from Theorem~\ref{thm:generalMultipartite}.} 

\begin{thm}\label{thm:generalMultipartite}
	Suppose $G$ is an $m$-partite graph with partite sets $A_1,A_2,\ldots ,A_m$.  Let
	\begin{equation*}
		d_j=\max\left\{|N_G(v)\cap A_j|:v\in \bigcup_{k=j+1}^mA_k\right\},
	\end{equation*}
	$p^*_m=1$, and $p^*_j$ be the unique solution in $(0,1)$ to $p=p^*_{j+1}(1-p)^{d_j}\text{ for all }j\in[m-1]$.
	Then $\chi^*_{_{DP}}(G)\leq\frac{1}{p^*_1}$.
\end{thm}  

%We prove Theorem~\ref{thm:generalMultipartite} by randomly selecting an independent transversal.  We analyze the number of vertices selected for this transversal from each list $L(u)$.  Our result follows from an application of the Chernoff-Hoeffding Bound, showing that there must exist at least one independent $b$-fold transversal for an appropriately large $b$.

The following corollary is an immediate consequence of Theorem~\ref{thm:generalMultipartite}.

\begin{cor}\label{cor:completeMultipartite}
	Suppose $G=K_{n_1,n_2,\ldots,n_m}$.  Let $p^*_m=1$ and $p^*_j$ be the solution in $(0,1)$ to $p=p^*_{j+1}(1-p)^{n_j}$ for all $j\in[m-1]$.  Then $\chi^*_{_{DP}}(G)=\chi^*_{_{DP}}(K_{n_1,n_2,\ldots,n_m})\leq\frac{1}{p^*_1}$.
\end{cor}

It follows from Theorem~\ref{thm: fracDP2} and the fact that $\chi(K_{1,m})=2$ for each $m \in \N$ that $\chi^*_{DP} (K_{1,m})=2$.  Applying Corollary~\ref{cor:completeMultipartite} to a wider range of bipartite and multipartite graphs, ordering the partite sets so $n_1\leq n_2\leq\cdots\leq n_m$, we get the results shown in Table~\ref{tab:compilation}.  For comparison, an upper bound on the DP-chromatic number is provided in the table (this upper bound is sharp for large enough $m$, and is also a sharp bound on the list chromatic number).  Note that the difference between the DP-chromatic number and our bound on the fractional DP-chromatic number can be made arbitrarily large.  Specifically, if we consider complete bipartite graphs, it is well-known that for each $n \in \N$, there is a $m_n \in \N$ such that $\chi_{_{DP}}(K_{n,m_n})=n+1$ (see e.g.,~\cite{M18}).  However, for $p^*_n$, the unique solution in $(0,1)$ to $p=(1-p)^n$, it can be verified that 
\begin{equation*}
    \lim_{n\to\infty}((n+1)-1/p^*_n)=\infty.
\end{equation*}
Using the notation of Corollary~\ref{cor:completeMultipartite}, it can be shown that for any $t \in \N$, there exist $n_1, n_2, \ldots, n_m \in \N$ such that $\chi_{_{DP}}(K_{n_1,n_2,\ldots,n_m}) - 1/p^*_1 > t$. This leads to the following question.
\begin{ques}
    Is there a function $f$ such that $f(\chi^*_{_{DP}}(G))\geq \chi_{_{DP}}(G)$ for all graphs $G$?
\end{ques}

\begin{table}[t]
	\centering
	\begin{tabular}{|c|c|c|}
		\hline
		Graph $G$&$\chi^*_{_{DP}}(G)$&$\chi_{_{DP}}(G)$\\
        \hline
        $K_{2,m}$&$\leq2.619$&$\leq3$\\
        $K_{3,m}$&$\leq3.148$&$\leq4$\\
        $K_{4,m}$&$\leq3.630$&$\leq5$\\
        \hline
    \end{tabular}
    \hspace{0.25in}
	\begin{tabular}{|c|c|c|}
		\hline
		Graph $G$&$\chi^*_{_{DP}}(G)$&$\chi_{_{DP}}(G)$\\
        \hline
        $K_{2,2,m}$&$\leq4.391$&$\leq5$\\
        $K_{2,3,m}$&$\leq4.946$&$\leq6$\\
        $K_{3,4,m}$&$\leq6.170$&$\leq8$\\
        \hline
	\end{tabular}
\caption{Upper bounds on $\chi^*_{_{DP}}(G)$ for any $m\in\N$ from Corollary~\ref{cor:completeMultipartite}.}
	\label{tab:compilation}
\end{table}

It is now important to mention a result from~\cite{BK18}.

\begin{thm} [\cite{BK18}] \label{thm: bipartiteupper}
If $G$ is a $d$-degenerate bipartite graph, then $\chi^*_{DP}(G) \leq (1 + o(1))d/\ln(d)$ as $d \rightarrow \infty$.
\end{thm}

Since, for $n\leq m$, the degeneracy of $K_{n,m}$ is $n$, Theorem~\ref{thm: bipartiteupper} provides a better upper bound on $\chi^*_{_{DP}}(K_{n,m})$ than Corollary~\ref{cor:completeMultipartite} when $n$ is sufficiently large.  In Section~\ref{lowerBound} we give a probabilistic argument to obtain the following lower bound on $\chi^*_{_{DP}}(K_{2,m})$.  

\begin{thm} \label{thm: fractionalK2m}
Suppose $G=K_{2,m}$ where $m \geq 3$.  Suppose $d \in (0,0.125)$ is chosen so that
$$\frac{(d+2)^{2/m} (d + 1)^{d + 1} (1-d)^{d-1}}{(d+2)(d^{2 d})} < 1.$$
Then, $2+d \leq \chi_{_{DP}}^*(G)$.
\end{thm}

For example, notice that when $m=15$ and $d=0.0959$, the inequality in the hypothesis is satisfied.  So, by Theorems~\ref{thm:generalMultipartite} and~\ref{thm: fractionalK2m}, we have that $2.0959 \leq \chi_{_{DP}}^*(K_{2,15}) \leq 2.619$.  Since we suspect the probabilistic argument from Theorem \ref{thm:generalMultipartite} to be close to the truth for large values of $m$, we think that the lower bound provided by Theorem~\ref{thm: fractionalK2m} can be improved by quite a bit for these $m$ values.  We also have that $2.025 \leq \chi_{_{DP}}^*(K_{2,3})\leq 2.619$ which provides another example of a graph whose fractional DP-chromatic number is larger than its list chromatic number since $2=\chi_\ell(K_{2,3}) < \chi_{_{DP}}^*(K_{2,3})$.

These results lead to a natural question.  
\begin{ques}\label{ques:asymptotic}
    Suppose $n\in\N$. If $p^*$ is the solution in $(0,1)$ to $p=(1-p)^n$, does $\chi_{_{DP}}^*(K_{n,m})\to1/p^*$ as $m\to\infty$?
\end{ques}
The answer to Question~\ref{ques:asymptotic} is clearly yes when $n=1$.  If the answer is yes for some $n\geq2$, then Theorem~\ref{thm:generalMultipartite} and Corollary~\ref{cor:completeMultipartite} give the asymptotically tight result as the second partite set grows large.  We could further ask whether, for some $n\geq2$, there is an $m\in\N$ where we achieve $\chi^*_{_{DP}}(K_{n,m})=1/p^*$?  If the answer to this follow-up question is yes it would give an example of a graph whose fractional DP-chromatic number is irrational, answering a question posed by Bernshteyn, Kostochka, and Zhu~\cite{BK18}. Question~\ref{ques:asymptotic} could be extended to complete multipartite graphs where the sizes of all but one of the partite sets are held constant.

\section{Odd Cycles} \label{oddCycle}

We now prove Theorem~\ref{thm: oddcycle}.

\begin{proof}
Suppose that the vertices of $G$ in cyclic order are: $v_1, v_2, \ldots, v_{2r+1}$.  Suppose that $\mathcal{H} = (L,H)$ is an arbitrary $(2r+1)$-fold cover of $G$.  We must show that there is an $(\mathcal{H},r)$-coloring of $G$.  We may assume that $E_H(L(u),L(v))$ is a perfect matching whenever $uv \in E(G)$ since adding additional edges to $H$ only makes it harder to find an $(\mathcal{H},r)$-coloring.

Clearly, $H$ is a 2-regular graph.  This means that $H$ can be decomposed into vertex disjoint cycles: $B_1, B_2, \ldots, B_p$.  The size of each of these cycles is a multiple of $2r+1$.  Let us suppose that $B_1, \ldots, B_l$ are even cycles and $B_{l+1}, \ldots, B_p$ are odd cycles (Note: we allow $l=0$ since it is possible that none of the cycles in our decomposition are even.  We also know the number of odd cycles in our decomposition must be odd since $|E(H)|=(2r+1)^2)$.  Clearly, $1 \leq p \leq 2r+1$, and $|L(v_i) \cap V(B_j)| \geq 1$ for each $i \in [2r+1]$ and $j \in [p]$.

Let $H^*$ be the graph obtained from $H$ as follows: for each $j \in [p]$ we delete a vertex $d_j \in L(v_j) \cap V(B_j)$.  We let $L'(v_j) = L(v_j) - \{d_j\}$ for each $j \in [p]$ and $L'(v_t)=L(v_t)$ for each $t>p$.  So, $H^*$ consists of $p$ vertex disjoint paths, and for each $j \in [p]$, let $P_j = B_j - \{d_j \}$.  Note that if $1 < j < 2r+1$, the endpoints of $P_j$ are in $L'(v_{j-1})$ and $L'(v_{j+1})$.  Also for each $j \in [p]$, $|V(P_j)| = (2r+1)k_j + 2r$ where $k_j$ is a nonnegative integer that is odd when $j \leq l$ and even when $j > l$.  It is easy to see:
$$|V(H^*)| = (2r+1)^2 - p = \sum_{j=1}^p ((2r+1)k_j + 2r) = 2rp + (2r+1)\sum_{j=1}^p k_j.$$
Thus, $\sum_{j=1}^p k_j = 2r+1 - p$.  Now, we name the vertices of each path of $H^*$.  Specifically, for $j \in [p]$ let the vertices of $P_j$ (in order) be: $a^j_{1}, a^j_{2}, \ldots, a^j_{(2r+1)k_j + 2r}$ so that $a^j_{1} \in L'(v_{j+1})$ if $j < 2r+1$ and $a^j_{1} \in L'(v_1)$ if $j=2r+1$.  We call a vertex $a^j_{m} \in V(H^*)$ \emph{odd} if $m$ is odd.  Let $S$ consist of all the odd vertices in $H^*$.  Clearly, $S$ is an independent set of $H$.  We claim that $|S \cap L'(v_i)| \geq r$ for each $i \in [2r+1]$.

In the case that $p=1$, $P_1$ is a path of length $(2r+1)^2-2$, and we have that for $i \in [2r+1]$, $|S \cap V(P_1) \cap L'(v_i)| = r+1$ when $i$ is even, and $|S \cap V(P_1) \cap L'(v_i)| = r$ when $i$ is odd.  In the case that $p = 2r+1$, each of $P_1, \ldots, P_{2r+1}$ is a path of length $2r-1$, and we have that for $i \in [2r+1]$, $|S \cap L'(v_i)| = r$.

So, we turn our attention to the case where $2 \leq p \leq 2r$.  For each $j \in [l]$ notice that $P_j$ is a path with an odd number of vertices.  So, when $j \in [l]$, $|S \cap V(P_j)|=(2r+1)(k_j+1)/2$.  Moreover, since $G$ is an odd cycle, for each $j \in [l]$ and $i \in [2r+1]$, we have that $|S \cap V(P_j) \cap L'(v_i)| = (k_j + 1)/2$.

Now, let $\mathcal{L} = \{l+1, l+3, \ldots, p-2 \}$ (Note: $\mathcal{L}$ is empty if $l+1 > p-2$ and $|\mathcal{L}| \leq r-1$.). For each $j \in \mathcal{L}$ we consider $P_j$ and $P_{j+1}$ together.  Note $P_j$ and $P_{j+1}$ are paths with an even number of vertices.  So, when $j \in \mathcal{L}$, $|S \cap (V(P_j) \cup V(P_{j+1}))|=(2r+1)(k_j + k_{j+1})/2 + 2r$.  Therefore, when $j \in \mathcal{L}$ and $i \in [2r+1]-\{j\}$, $|S \cap (V(P_j) \cup V(P_{j+1})) \cap L'(v_i)| = 1 + (k_j + k_{j+1})/2$, and for each $j \in \mathcal{L}$, $|S \cap (V(P_j) \cup V(P_{j+1})) \cap L'(v_j)| = (k_j + k_{j+1})/2$ (since $a^j_{1} \in L'(v_{j+1})$ and $a^{j+1}_{(2r+1)k_{j+1} + 2r} \in L'(v_j)$).  Thus, for $i \in [2r+1] - \mathcal{L}$, we have:
\begin{align*}
\left| \bigcup_{j=1}^{p-1} (V(P_j) \cap S \cap L'(v_i)) \right | &= \sum_{j=1}^l \frac{k_i+1}{2} + \sum_{j \in \mathcal{L}} \frac{k_j + k_{j+1} +2}{2} \\
&= \frac{1}{2} \sum_{j=1}^{p-1} (k_i+1) \\
&= \frac{p-1}{2} + \frac{1}{2} \sum_{j=1}^{p-1} k_i \\
&= \frac{p-1+2r+1-p-k_p}{2} = r - \frac{k_p}{2}.
\end{align*}
Similarly, for $i \in \mathcal{L}$, we have:
$$\left | \bigcup_{j=1}^{p-1} (V(P_j) \cap S \cap L'(v_i)) \right | = r -1 - \frac{k_p}{2}.$$
It is easy to see that $|S \cap V(P_p) \cap L'(v_i)| \geq k_p/2$ for each $i \in [2r+1]$.  If $|\mathcal{L}| \geq 1$, notice that $a^p_{(2r+1)k_p + 2r} \in L'(v_{p-1})$.  So, when $|\mathcal{L}| \geq 1$, the last odd vertex in $V(P_p)$ is in $L'(v_{p-2})$.  So, we know that for each $i \in \mathcal{L}$,
$$|S \cap V(P_p) \cap L'(v_i)| = \frac{k_p}{2} + 1.$$
It follows that $|S \cap L'(v_i)| \geq r$ for each $i \in [2r+1]$ which implies that $S$ is an $(\mathcal{H},r)$-coloring of $G$.
\end{proof}

\section{Multipartite Graphs}
\label{multipartiteGraphs}

We now work toward proving our bounds on the fractional DP-chromatic number of complete bipartite graphs.  In order to prove Theorem \ref{thm:generalMultipartite} we will use two lemmas.
\begin{lem}\label{lem:independentSet}%Define $\sim B()$ in the intro.
	Suppose $G$ is a graph where $\{U,W\}$ is a partition of $V(G)$, $d=\max_{w\in W}|N_G(w)\cap U|$, $p\in(0,1)$, and $\epsilon>0$. There is an $N\in\N$ such that for any $a\geq N$, if $\mathcal{H}=(L,H)$ is an $a$-fold cover of $G$, then there must exist $U'\subseteq\bigcup_{u\in U}L(u)$ and $W'\subseteq\bigcup_{w\in W}L(w)$ such that the following three conditions hold:
    \begin{enumerate}
        \item $|E_H(U',W')|=0$,

        \item $|U'\cap L(u)|\geq  (p-\epsilon)a$ for all $u\in U$,

        \item $|W'\cap L(w)|\geq ((1-p)^d-\epsilon)a$ for all $w\in W$.
    \end{enumerate}
\end{lem}
\begin{proof}
	Let $n=|V(G)|$.  We will now give a random procedure for constructing $C\subseteq\bigcup_{u\in U}L(u)$ and $D\subseteq\bigcup_{w\in W}L(w)$ which we will use to guarantee the existence of sets $U'$ and $W'$ as described in the statement.  As in the proof of Theorem~\ref{thm: oddcycle}, we may assume that $E_H(L(u),L(v))$ is a perfect matching whenever $uv \in E(G)$.
	
	For each $u\in U$ and each $x\in L(u)$ include $x$ in $C$ independently with probability $p$.  For each $w\in W$ and each $y\in L(w)$ include $y$ in $D$ if $y$ is not adjacent to any of the vertices in $C$.  
	
	The probability that, for any $w\in W$, a vertex from $L(w)$ is included in $D$ is at least $(1-p)^d$.  For each $u\in U$, let $X_u$ be the random variable that is the number of vertices included in $C$ from $L(u)$.  For each $w\in W$, let $Y_w$ be the random variable that is the number of vertices included in $D$ from $L(w)$.  Let $E_{X_u}$ be the event that $X_u\geq  (p-\epsilon)a$ and $E_{Y_w}$ be the event that $Y_w\geq ((1-p)^d-\epsilon)a$.  Since $X_u\sim B(a,p)$ and $Y_w\sim B(a,p_w)$ for some $p_w\geq (1-p)^d$, we know that for any $u\in U$ and $w\in W$
	\begin{align*}
		&\mathbb{P}(\overline{E_{X_u}})=\mathbb{P}(X_u<(p-\epsilon)a)\leq e^{-2a\epsilon^2},\\
		&\mathbb{P}(\overline{E_{Y_w}})=\mathbb{P}(Y_w<((1-p)^d-\epsilon)a)\leq\mathbb{P}(Y_w<(p_w-\epsilon)a)\leq e^{-2a\epsilon^2}.
	\end{align*}
	By the union bound, we know the probability that $|C\cap L(u)|\geq  (p-\epsilon)a$ for each $u\in U$ and that $|D\cap L(w)|\geq ((1-p)^d-\epsilon)a$ for each $w\in W$ is
	\begin{equation*}
		\mathbb{P}\left(\left(\bigcap_{u\in U}E_{X_u}\right)\bigcap\left(\bigcap_{w\in W}E_{Y_w}\right)\right)=1-\mathbb{P}\left(\left(\bigcup_{u\in U}\overline{E_{X_u}}\right)\bigcup\left(\bigcup_{w\in W}\overline{E_{Y_w}}\right)\right)\geq1-ne^{-2a\epsilon^2}.
	\end{equation*}
	
	We now show that this probability is positive for large enough $a$.  Let $N$ be any integer larger than $\ln (n)/(2\epsilon^2)$.  For any $a\geq N$ we see
	\begin{equation*}
		1-ne^{-2a\epsilon^2}>1-ne^{-2\left(\ln(n)/(2\epsilon^2)\right)\epsilon^2}=0.
	\end{equation*}
	Therefore, there must be sets $U'$ and $W'$ as described in the statement.
\end{proof}

\begin{lem}\label{lem:inductionStep}%Look at the notation for (a,b)-DP colorable when $b\leq0$.  Define the appropriate subcover in the introduction and check it is correct here.  Fix names of N
	Suppose $G$ is a graph where $\{U,W\}$ is a partition of $V(G)$ and $U$ is an independent set of vertices.  Let $d=\max_{w\in W}|N_G(w)\cap U|$.  Fix $p'\in(0,1]$ and suppose $G[W]$ has the property that for any $\epsilon'\in(0,p')$ there is an $N_{\epsilon'}\in\N$ such that $G[W]$ is $(a',\lfloor(p'-\epsilon')a'\rfloor)$-DP-colorable for all $a'\geq N_{\epsilon'}$.
 
    Suppose $p^*$ is the unique element in $(0,1)$ satisfying $p=p'(1-p)^d$.  Then $G$ has the property that for any $\epsilon\in(0,p^*)$ there is an $N\in\N$ such that $G$ is $(a,\lfloor(p^*-\epsilon)a\rfloor)$-DP-colorable for all $a\geq N$.  Consequently, $\chi^*_{_{DP}}(G)\leq1/p^*$.
\end{lem}
\begin{proof}
	Consider an arbitrary $a$-fold cover $\mathcal{H}=(L,H)$ of $G$, $\epsilon'\in(0,p')$, and $\epsilon^*\in(0,p^*)$.  
    
    By Lemma \ref{lem:independentSet}, there exists some $N_{\epsilon^*}$ such that if $a\geq N_{\epsilon^*}$ then we must be able to get a set $U'\subseteq\bigcup_{u\in U}L(u)$ that contains a $\lfloor(p^*-\epsilon^*)a\rfloor$-fold transversal of $\mathcal{H}_U$, and a set $W'\subseteq\bigcup_{w\in W}L(w)$ that contains a $\lfloor((1-p^*)^d-\epsilon^*)a\rfloor$-fold transversal of $\mathcal{H}_W$.  Note that $\lfloor((1-p^*)^d-\epsilon^*)a\rfloor=\lfloor(p^*/p'-\epsilon^*)a\rfloor$.  Moreover, $|E_H(U',W')|=0$.
	
	By the lemma statement, there is an $N_{\epsilon'}\in\N$ such that $G[W]$ is $(a',\lfloor(p'-\epsilon')a'\rfloor)$-DP-colorable for all $a'\geq N_{\epsilon'}$.  
 
    For each $w\in W$ let $L'(w)=W'\cap L(w)$.  If $a\geq N_{\epsilon^*}$ then we know that $(L',H[W'])$ contains a $\lfloor(p^*/p'-\epsilon^*)a\rfloor$-fold cover of $G[W]$.   So $(L',H[W'])$ must have an independent $\lfloor(p'-\epsilon')\lfloor(p^*/p'-\epsilon^*)a\rfloor\rfloor$-fold transversal if $a$ also satisfies $\lfloor(p^*/p'-\epsilon^*)a\rfloor\geq N_{\epsilon'}$.  Let $N_{\epsilon'}'=\lceil(N_{\epsilon'}+1)/(p^*/p'-\epsilon^*)\rceil$.  Note that $\lfloor(p^*/p'-\epsilon^*)a\rfloor\geq N_{\epsilon'}$ is satisfied if $a\geq N_{\epsilon'}'$.
	
	Notice
	\begin{equation*}
		\lfloor(p'-\epsilon')\lfloor(p^*/p'-\epsilon^*)a\rfloor\rfloor\geq(p'-\epsilon')(p^*/p'-\epsilon^*)a-2\geq(p^*-p'\epsilon^*-\epsilon'p^*/p')a-2.
	\end{equation*}
	Given $\epsilon\in(0,p^*)$ we can fix an $\epsilon'\in(0,p')$ and fix an $\epsilon^*\in(0,\epsilon)$ such that
	\begin{equation*}
		\epsilon>p'\epsilon^*+\epsilon'p^*/p'.
	\end{equation*}
	Therefore, there must be some $N^*\in\N$ such that for all $a\geq N^*$, 
	\begin{equation*}
		\lfloor(p'-\epsilon')\lfloor(p^*/p'-\epsilon^*)a\rfloor\rfloor\geq(p^*-p'\epsilon^*-\epsilon'p^*/p')a-2\geq(p^*-\epsilon)a.
	\end{equation*}
	
	If $a\geq\max\{N_{\epsilon^*},N_{\epsilon'}',N^*\}$ we know that $W'$ must contain an independent $\lfloor(p^*-\epsilon)a\rfloor$-fold transversal of $\mathcal{H}_W$.  Call this $\lfloor(p^*-\epsilon)a\rfloor$-fold transversal $T_W$.  And since we chose $\epsilon^*$ to be less than $\epsilon$, we also know that $U'$ must contain a $\lfloor(p^*-\epsilon)a\rfloor$-fold transversal of $\mathcal{H}_U$.  Call this $\lfloor(p^*-\epsilon)a\rfloor$-fold transversal $T_U$.  Since $U$ is an independent set of vertices, we know the vertices in $T_U$ form an independent transversal of $\mathcal{H}_U$.  We know $|E_H(U',W')|=0$.  Therefore, $T_U\cup T_W$ is an independent $\lfloor(p^*-\epsilon)a\rfloor$-fold transversal of $\mathcal{H}$ and $G$ is $(a,\lfloor(p^*-\epsilon)a\rfloor)$-DP-colorable.

    For each $\epsilon>0$ suppose $M_\epsilon$ satisfies: $G$ is $(a,\lfloor(p^*-\epsilon)a\rfloor)$-DP-colorable for any $a\geq M_\epsilon$, consider the sequences $\epsilon_k=1/k$ and $a_k=\max\{k,M_{\epsilon_k}\}$.   For sufficiently large $k$, guaranteeing $(p^*-1/k)k-1>0$, we know
    \begin{equation*}
        \chi^*_{_{DP}}(G)\leq\frac{a_k}{\lfloor(p^*-\epsilon_k)a_k\rfloor}\leq\frac{a_k}{(p^*-\epsilon_k)a_k-1}=\frac{1}{p^*-\epsilon_k-1/a_k}\leq\frac{1}{p^*-1/k-1/k}=\frac{1}{p^*-2/k}.
    \end{equation*}
    By taking the limit as $k\to\infty$, it follows that $\chi^*_{_{DP}}(G)\leq1/p^*$.
\end{proof}

We are now ready to prove Theorem \ref{thm:generalMultipartite}.\begin{customthm}{\bf\ref{thm:generalMultipartite}}
	Suppose $G$ is an $m$-partite graph with partite sets $A_1,A_2,\ldots ,A_m$.  Let
	\begin{equation*}
		d_j=\max\left\{|N_G(v)\cap A_j|:v\in \bigcup_{k=j+1}^mA_k\right\},
	\end{equation*}
	$p^*_m=1$, and $p^*_j$ be the unique solution in $(0,1)$ to $p=p^*_{j+1}(1-p)^{d_j}\text{ for all }j\in[m-1]$.
	Then $\chi^*_{_{DP}}(G)\leq\frac{1}{p^*_1}$.
\end{customthm}
\begin{proof}
	We will prove the stronger statement: $G$ has the property that for any $\epsilon\in(0,p^*_1)$ there is an $N\in\N$ such that $G$ is $(a,\lfloor(p^*_1-\epsilon)a\rfloor)$-DP-colorable for all $a\geq N$, and consequently $\chi^*_{_{DP}}(G)\leq1/p^*_1$.  This follows from induction on the number of partite sets $m$ using Lemma \ref{lem:inductionStep}.  For the base case, when $m=2$, $G$ is the bipartite graph with partite sets $A_1$ and $A_2$.  Let $U=A_1$ and $W=A_2$.  Since $G[W]$ is an independent set, it has the property that every $a'$-fold cover of $G[W]$ has an independent $a'$-fold transversal (satisfying the requirement for Lemma~\ref{lem:inductionStep} with $p'=1$).  Let $d$ be $d_1=\max_{w\in W}|N_G(w)\cap U|$ and $p^*_1$ be the unique element in $(0,1)$ satisfying $p=1(1-p)^d$. Applying Lemma~\ref{lem:inductionStep} completes the base case.

    Next, we assume that our result holds for any given $k$-partite graph for some fixed $k\geq2$.  Consider $G$, a $(k+1)$-partite graph with partite sets $A_1,A_2,\ldots,A_{k+1}$.  Let $U=A_1$ and $W=\bigcup_{i=2}^{k+1}A_i$.  Notice by the induction hypothesis that $G[W]$ is a $k$-partite graph with partite sets $A_2,\ldots,A_{k+1}$ that satisfies the hypothesis of Lemma~\ref{lem:inductionStep} with $p'=p^*_2$.  Applying Lemma~\ref{lem:inductionStep} with $d=d_1$ and $p^*=p_1^*$ shows that the statement holds for $G$.

    Therefore, our result holds for any $m$-partite graph with $m\geq2$ by induction.
 
 %Since the graph induced by the second partite set is an independent set of vertices and every independent set of vertices has the property that any $a$-fold cover has an independent $a$-fold transversal, then, the $p'$ value for the subgraph induced by the second partite set is $1$.  Applying Lemma \ref{lem:inductionStep} proves the result for the base case.  Now, we can perform induction on the value of $m$, the number of partite sets, using Lemma \ref{lem:inductionStep}.
\end{proof}

Corollary~\ref{cor:completeMultipartite} follows immediately from Theorem~\ref{thm:generalMultipartite}.  

\section{Lower Bound for Complete Bipartite Graphs}
\label{lowerBound}

From this point forward, when considering a copy of the complete bipartite graph $K_{n,m}$, we will always assume that the partite sets are $A = \{v_1, \ldots, v_n \}$ and $B = \{u_1, \ldots, u_m \}$.  We will now use a probabilistic argument to prove Theorem~\ref{thm: fractionalK2m}.

\begin{proof}
Throughout this proof suppose $m \in \N$ is fixed and $m \geq 3$.  Since $G$ contains more than one even cycle we know that $\chi_{_{DP}}^*(G) > 2$ by Theorem~\ref{thm: fracDP2}.  Our goal for this proof is to show that $\chi_{_{DP}}^*(G) \geq 2+d$.  So, suppose that $a$ and $t$ are arbitrary natural numbers such that $2 < a/t \leq 2+d$.  Also, let $r = a/t$ and $\delta = r-2$ so that $\delta \in (0,d]$.  To prove the result, it is sufficient to show that $G$ is not $(a,t)$-DP-colorable.

We form an $a$-fold cover $(L,H)$ of $G$ by the following (partially random) process. We begin by letting $L(v_i) = \{(v_i, l) : l \in [a] \}$ and $L(u_j) = \{(u_j, l) : l \in [a] \}$ for each $i \in [2]$ and $j \in [m]$.  Let the graph $H$ have vertex set
$$\left (\bigcup_{i=1}^{2} L(v_i) \right ) \bigcup \left ( \bigcup_{j=1}^{m} L(u_j) \right ).$$
\noindent Let $L(v)$ be an independent set of vertices in $H$ for each $v \in V(G)$.  Finally, for each $i \in [2]$ and $j \in [m]$, uniformly at random choose a perfect matching between $L(v_i)$ and $L(u_j)$ from the $a!$ possible perfect matchings.  It is easy to see that $\mathcal{H}=(L,H)$ is an $a$-fold cover of $G$.

We want to show that with positive probability there is no $(\mathcal{H}, t)$-DP-coloring of $G$.  For $i=1,2$, let $\mathcal{A}_i$ be the set of $t$-element subsets of $L(v_i)$.  We say $(A_1,A_2)\in \mathcal{A}_1\times\mathcal{A}_2$ is \emph{good for $u_j$} if $|L(u_j)\setminus (N_H(A_1\cup A_2))|\geq t$, meaning there is a $t$-element subset of $L(u_j)$ that is independent of $A_1\cup A_2$. We know we can find a $(\mathcal{H}, t)$-coloring of $G$ if $(A_1, A_2) \in \mathcal{A}_1 \times \mathcal{A}_2$ is good for each vertex in $\{u_j, : j \in [m] \}$.  Let $E_j$ be the event that $(A_1, A_2)$ is good for $u_j$.  In order for $E_j$ to occur we need at least $3t-a$ of the vertices in $N_{H}(A_1) \cap L(u_j)$ to also be in $N_{H}(A_2) \cap L(u_j)$.  So,
\begin{align*}
P[E_j] = \binom{a}{t}^{-1} \sum_{i=3t-a}^{t} \binom{t}{i} \binom{a-t}{t-i} &= \binom{a}{t}^{-1} \sum_{i=3t-a}^{t} \binom{t}{t-i} \binom{a-t}{t-i} \\
&= \binom{a}{t}^{-1} \sum_{i=0}^{a-2t} \binom{t}{i} \binom{a-t}{i}.
\end{align*}

Since $r \leq 2.125 < 2.5$, it easily follows that $a-2t < t/2$ and $a-2t < (a-t)/2$.  Using a well known bound on the partial sum of binomial coefficients (see~\cite[Thm $3.1$]{G14}), we obtain:

\begin{align*}
P[E_j] &= \binom{a}{t}^{-1} \sum_{i=0}^{a-2t} \binom{t}{i} \binom{a-t}{i} \\
&\leq \binom{a}{t}^{-1} \left ( \sum_{i=0}^{a-2t} \binom{t}{i} \right) \left( \sum_{i=0}^{a-2t} \binom{a-t}{i} \right ) \\
& \leq \binom{a}{t}^{-1} \left ( \frac{t}{a-2t} \right)^{a-2t} \left ( \frac{t}{3t-a} \right)^{3t-a} \left ( \frac{a-t}{a-2t} \right )^{a-2t} \left ( \frac{a-t}{t} \right )^{t} \\
& = \binom{a}{t}^{-1} \left ( \frac{r-1}{(r-2)^2} \right)^{t(r-2)} \left ( \frac{1}{3-r} \right)^{t(3-r)} \left ( r-1 \right )^{t} \\
& = \binom{a}{t}^{-1} \left ( \frac{\delta + 1}{\delta^2} \right)^{t\delta} \left ( \frac{1}{1-\delta} \right)^{t(1-\delta)} \left ( \delta + 1 \right )^{t}.
\end{align*}

The probability $(A_1, A_2)$ is good for each vertex in $\{u_j : j \in [m] \}$ is then $(P[E_1])^m$. Since $|\mathcal{A}_1 \times \mathcal{A}_2| = \binom{a}{t}^2$, we can guarantee the existence of an $a$-fold cover, $\mathcal{H}^*$, for $G$ such that there is no $(\mathcal{H}^*, t)$-coloring of $G$ if
$$\binom{a}{t}^2 (P[E_1])^m < 1.$$
Using a well known bound on binomial coefficients, we compute
\begin{align*}
\binom{a}{t}^2 (P[E_1])^m & \leq \binom{a}{t}^{2-m} \left ( \frac{\delta + 1}{\delta^2} \right)^{mt\delta} \left ( \frac{1}{1-\delta} \right)^{mt(1-\delta)} \left ( \delta + 1 \right )^{mt} \\
& \leq \left (\frac{t}{a} \right)^{t(m-2)} \left ( \frac{\delta + 1}{\delta^2} \right)^{mt\delta} \left ( \frac{1}{1-\delta} \right)^{mt(1-\delta)} \left ( \delta + 1 \right )^{mt} \\
& = \left [ \left (\frac{1}{r} \right)^{1-2/m} \left ( \frac{\delta + 1}{\delta^2} \right)^{\delta} \left ( \frac{1}{1-\delta} \right)^{(1-\delta)} \left ( \delta + 1 \right )     \right ]^{mt} \\
& = \left [ \frac{(\delta+2)^{2/m} (\delta + 1)^{\delta + 1} (1-\delta)^{\delta-1}}{(\delta+2)(\delta^{2 \delta})} \right]^{mt}.
\end{align*}
Thus, to prove the desired it suffices to show that:
$$ \frac{(\delta+2)^{2/m} (\delta + 1)^{\delta + 1} (1-\delta)^{\delta-1}}{(\delta+2)(\delta^{2 \delta})} < 1.$$
Consider the function $f: (0, 1) \rightarrow (0, \infty)$ given by $f(x) = \frac{(x+2)^{2/m} (x + 1)^{x + 1} (1-x)^{x-1}}{(x+2)(x^{2 x})}$.  It is easy to verify that $f$ is increasing on $(0, 0.5)$.  So, since $0<\delta \leq d < 0.5$,
$$\frac{(\delta+2)^{2/m} (\delta + 1)^{\delta + 1} (1-\delta)^{\delta-1}}{(\delta+2)(\delta^{2 \delta})} = f(\delta) \leq f(d) = \frac{(d+2)^{2/m} (d + 1)^{d + 1} (1-d)^{d-1}}{(d+2)(d^{2 d})} < 1$$
as desired.
\end{proof}

Notice that in our argument above the upper bound:  $\binom{a}{t}^{-1} \left ( \sum_{i=0}^{a-2t} \binom{t}{i} \right) \left( \sum_{i=0}^{a-2t} \binom{a-t}{i} \right )$ used for $\binom{a}{t}^{-1} \sum_{i=0}^{a-2t} \binom{t}{i} \binom{a-t}{i}$ is a fairly weak upper bound.  So, our result may be able to be improved significantly with a better upper bound on $\binom{a}{t}^{-1} \sum_{i=0}^{a-2t} \binom{t}{i} \binom{a-t}{i}$.  For a concrete application of Theorem~\ref{thm: fractionalK2m}, notice that when $m=15$ and $d=0.0959$, the inequality in the hypothesis is satisfied.  So, by Corollary~\ref{cor:completeMultipartite} and Theorem~\ref{thm: fractionalK2m}, we have that $2.0959 \leq \chi_{_{DP}}^*(K_{2,15}) \leq 2.619$.

\vspace{5mm}

{\bf Acknowledgment.}  The authors would like to thank Doug West for helpful conversations regarding this paper.

\end{document}